\documentclass[smallextended]{article}       
\usepackage[utf8]{inputenc}
\usepackage{graphicx}
\DeclareGraphicsExtensions{.bmp} 
\usepackage{subfigure}
\usepackage{epstopdf}
\usepackage{cite}
\usepackage[centertags]{amsmath}
\usepackage{amsfonts,amssymb,newlfont,makeidx,listings,xcolor}
\usepackage[colorlinks=true,urlcolor=blue,linkcolor=blue,pdfborder={0 0 0}]{hyperref}
\usepackage[left=20mm,top=20mm,right=20mm,bottom=20mm]{geometry}

\newtheorem{theorem}{Theorem}[section]
\newtheorem{corollary}{Corollary}[section]
\newtheorem{remark}{Remark}[section]

\newenvironment{proof}{\textit{Proof}:}{\hfill$\square$}
\numberwithin{equation}{section}

\begin{document}

	\title{Quarter-symmetric connection on an almost Hermitian manifold and on a K\"ahler manifold}
	\date{}
	\author{\textbf{Milan Lj. Zlatanovi\'c${}^1$, Miroslav D. Maksimovi\'c${}^2$}}
	\maketitle

	\noindent ${}^1${University of Ni\v s, Faculty of Sciences and Mathematics, Department of Mathematics, Ni\v s, Serbia,} \\ mail: {zlatmilan@yahoo.com}
	
	\vskip0.25cm
	\noindent ${}^2${University of Pri\v stina in Kosovska Mitrovica, Faculty of Sciences  and Mathematics, Department of Mathematics, Kosovska Mitrovica, Serbia,} mail: miroslav.maksimovic@pr.ac.rs (corresponding)

	\begin{abstract}\small 
		The paper observes an almost Hermitian manifold as an example of a generalized Riemannian manifold and examines the application of a quarter-symmetric connection on the almost Hermitian manifold. The almost Hermitian manifold with quarter-symmetric connection preserving the generalized Riemannian metric is actually the K\"ahler manifold. Observing the six linearly independent curvature tensors with respect to the quarter-symmetric connection, we construct tensors that do not depend on the quarter-symmetric connection generator. One of them coincides with the Weyl projective curvature tensor of symmetric metric $g$. Also, we obtain the relations between the Weyl projective curvature tensor and the holomorphically projective curvature tensor. Moreover, we examine the properties of curvature tensors when some tensors are hybrid. 
		
		\vskip0.25cm
		\noindent\textbf{Keywords}: Almost Hermitian manifold, curvature tensors, hybrid tensor, K\"{a}hler manifold, quarter-symmetric connection, torsion tensor
		
		\vskip0.25cm
		\noindent\textbf{MSC 2020}: 53B05, 53B35, 53C05, 53C15.
	\end{abstract}

	\section{Introduction}
	
	The paper deals with a non-symmetric linear connection, i.e. investigates the linear connection with the torsion tensor. In \cite{ivanov2016}, the authors discussed linear connections in a generalized Riemannian manifold. Among other things, they studied connections with a totally skew-symmetric torsion tensor and connections with the Einstein metricity condition (see also \cite{IZ2020}). Paper \cite{zlatanovic2021} studied the semi-symmetric metric connection and properties of the curvature tensor and determined the relations between Weyl projective curvature tensor, conformal curvature tensor, and concircular curvature tensor. S. Golab in \cite{golab1975} defined the quarter-symmetric connection as a generalization of the semi-symmetric connection. After the initial work, the theory of quarter-symmetric connection was expanded by many authors in various manifolds (for instance see \cite{bulut2019, chaturvedi2015b, chaubey2008, de2020, han2013, khan2020, tang2018}). In paper \cite{tripathi2008}, M. Tripathi introduced a new linear connection with torsion tensor in a Riemannian manifold, which generalizes several semi-symmetric and quarter-symmetric connections.

	In paper \cite{zlma2023}, the authors defined a quarter-symmetric generalized metric connection on a generalized Riemannian manifold as a connection that preserves the generalized (non-symmetric) Riemannian metric. The paper determined relations for curvature tensors and studied their skew-symmetric and cyclic-symmetric properties. This paper is a continuation of the above research, and now we will deal with the application of the quarter-symmetric connection on the almost Hermitian manifold. We will show that the almost Hermitian manifold with quarter-symmetric generalized metric connection is actually a K\"ahler manifold.  Accordingly, we will find some identities for the holomorphically projective curvature tensor and the Weyl projective curvature tensor. The \textit{holomorphically projective curvature tensor} given by equation
	\begin{equation}\label{eq:holomorphicallytensorKahler}
		\begin{split}
			\overset{g}{P} (X,Y)Z  =  \overset{g}{R} (X,Y)Z  & + \frac{1}{n+2} ( \overset{g}{R}ic (X,Z)Y - \overset{g}{R}ic (Y,Z)X )
			\\ &  - \frac{1}{n+2} (\overset{g}{R}ic (X,AZ)AY   -  \overset{g}{R}ic (Y,AZ)AX + 2\overset{g}{R}ic (X,AY)AZ)
		\end{split}
	\end{equation}
	is invariant under holomorphically projective mapping between two K\"ahler manifolds (see \cite{prvanovic2005,yano1965}). Such mapping is a natural generalization of geodesic mapping. The \textit{Weyl projective curvature tensor} given by equation
	\begin{equation}\label{eq:Weylptensor}
		\overset{g}{W} (X,Y)Z =  \overset{g}{R} (X,Y)Z +  \frac{1}{n-1}  ( \overset{g}{R}ic (X,Z)Y - \overset{g}{R}ic (Y,Z)X)
	\end{equation}
	is invariant under geodesic mapping between two Riemannian manifolds (for instance see \cite{mikes2015}). 

	\section{Preliminaries}
	
	Let $(\mathcal{M}, G=g+F)$ be a generalized Riemannian manifold, where $\mathcal{M}$ is an $n$-dimensional differentiable manifold, $G$ is a non-symmetric (0,2) tensor (the so-called generalized Riemannian metric), $g$ is the symmetric part of $G$ and $F$ is the skew-symmetric part of $G$. Tensor $A$ is defined as a tensor associated with tensor $F$, i.e. 
	\begin{equation}\label{eq:FgA}
		F(X,Y) = g(AX,Y).
	\end{equation}
	The quarter-symmetric connection $\overset{1}{\nabla}$ preserving generalized Riemannian metric $G$, $\overset{1}{\nabla} G=0$, is called \textit{quarter-symmetric generalized metric connection}  (i.e. \textit{quarter-symmetric $G$-metric connection}), and it is determined by equations (see \cite{zlma2023})
	\begin{equation}\label{eq:Q-S-Gmetric}
		\overset{1}{\nabla}_{X} Y = \overset{g}{\nabla}_{X} Y -\pi(X) A Y 
	\end{equation}
	and
	\begin{align}
		(\overset{1}{\nabla}_{X} g )(Y,Z) & =0,
		\\
		\label{eq:covderivofA-QSGMC}
		(\overset{1}{\nabla}_{X} A )Y   =(\overset{g}{\nabla}_{X} A)Y  &= 0, 
	\end{align}
	where $\pi$ is a 1-form associated with vector field $P$, i.e. $\pi(X)=g(X,P)$, and $\overset{g}{\nabla}$ is a Levi-Civita connection. A 1-form $\pi$ is called the \textit{generator} of that connection. The covariant derivative of generator $\pi$ is given by equation
	\begin{equation*}
		(\overset{1}{\nabla}_{X} \pi )(Y)  = (\overset{g}{\nabla}_{X} \pi )(Y) + \pi(X)\pi(AY).
	\end{equation*}
	The torsion tensor of connection $\overset{1}{\nabla}$ is given by equation
	\begin{equation*}
		\overset{1}{T}(X,Y)= \pi(Y) A X - \pi(X) A Y,
	\end{equation*}
	from which it follows
	\begin{equation*}
		\overset{1}{T}(X,Y,Z)= g(\overset{1}{T}(X,Y),Z)= \pi(Y) F (X,Z) - \pi(X) F (Y,Z).
	\end{equation*}
	
	The following statement gives known relations between curvature tensors $\overset{\theta}{R}$, $\theta =0,1,\dots, 5$, and Riemannian curvature tensor $\overset{g}{R}$.
	
	\begin{theorem}\label{thm:curvaturetensorsofQSGMC2}\cite{zlma2023}
		Let $(\mathcal{M}, G=g+F)$ be a generalized Riemannian manifold with the quarter-symmetric $G$-metric connection (\ref{eq:Q-S-Gmetric}). The curvature tensors $\overset{\theta}{R}$, $\theta =0,1,\dots,5$ and Riemannian curvature tensor $\overset{g}{R}$ satisfy the following relations
		\begin{align}
			\label{eq:R0XYZgm1}
			\begin{split}
				\overset{0}{R} (X,Y)Z = & \overset{g}{R} (X,Y)Z - \frac{1}{2}   ( \overset{0}{D} (X,Y) - \overset{0}{D} (Y,X)  )A Z    -  \frac{1}{2} \overset{0}{D} (X,Z) A Y + \frac{1}{2} \overset{0}{D} (Y,Z) A X 
				\\ & - \frac{1}{4} \pi (Z) (\pi (Y) A^2 X - \pi (X) A^2 Y  ),
			\end{split} \\
			\label{eq:R1XYZgm}
			\overset{1}{R} (X,Y)Z = & \overset{g}{R} (X,Y)Z  - \overset{1}{D} (X,Y)  A Z,
		\end{align}
		\begin{align}
			\label{eq:R2XYZgm1}
			\begin{split}
				\overset{2}{R} (X,Y)Z = & \overset{g}{R} (X,Y)Z  -  \overset{2}{D} (X,Z)A Y +  \overset{2}{D} (Y,Z)AX,
			\end{split} \\
			\label{eq:R3XYZgm1}
			\begin{split}
				\overset{3}{R} (X,Y)Z = & \overset{g}{R} (X,Y)Z  -  \overset{2}{D} (X,Y)A Z +  \overset{3}{D} (Y,Z)AX,
			\end{split} \\
			\label{eq:R4XYZgm1}
			\begin{split}
				\overset{4}{R} (X,Y)Z = & \overset{g}{R} (X,Y)Z  -  \overset{3}{D} (X,Y)A Z +  \overset{3}{D} (Y,Z)AX - \pi (Z) (\pi (Y) A^2 X - \pi (X) A^2 Y  ),
			\end{split} \\
			\label{eq:R5XYZgm1}
			\begin{split}
				\overset{5}{R} (X,Y)Z = & \overset{g}{R} (X,Y)Z - \frac{1}{2}  ( \overset{2}{D} (X,Y) - \overset{3}{D} (Y,X)  )A Z    -  \frac{1}{2} \overset{3}{D} (X,Z) A Y + \frac{1}{2} \overset{2}{D} (Y,Z) A X
				\\ & + \frac{1}{2} \pi (Y) (\pi (X) A^2 Z - \pi (Z) A^2 X  ),
			\end{split}
		\end{align}
		where
		\begin{align}
			\label{eq:D0}
			\overset{0}{D} (X,Y) & = (\overset{g}{\nabla}_{X} \pi)(Y) + \frac{1}{2} \pi(X)\pi(AY) + \frac{1}{2} \pi(Y)\pi(AX), \\
			\label{eq:D1}
			\overset{1}{D} (X,Y) & =  (  \overset{g}{\nabla}_{X} \pi  )(Y) - ( \overset{g}{\nabla}_{Y}\pi  )(X), \\
			\label{eq:D2}
			\overset{2}{D} (X,Y) & = (\overset{g}{\nabla}_{X} \pi)(Y) + \pi(Y)\pi(AX), \\
			\label{eq:D3}
			\overset{3}{D} (X,Y) & = (\overset{g}{\nabla}_{X} \pi)(Y) + \pi(X)\pi(AY)= (\overset{1}{\nabla}_{X} \pi)(Y).
		\end{align}
	\end{theorem}
	
	The corresponding (0,4) curvature tensors are defined by relations
	\begin{equation*}
		\overset{\theta}{R} (X,Y,Z,W) = g(\overset{\theta}{R} (X,Y)Z,W), \theta=0,1,\dots,5 \;\;\; \mbox{and} \;\;\; \overset{g}{R} (X,Y,Z,W) = g(\overset{g}{R} (X,Y)Z,W).
	\end{equation*}
	The corresponding Ricci tensors are defined by relations
	\begin{equation*}
		\overset{\theta}{R}ic (Y,Z) = Trace\{ X\rightarrow \overset{\theta}{R} (X,Y)Z \}, \theta=0,1,\dots,5 \;\;\; \mbox{and} \;\;\; \overset{g}{R}ic (Y,Z) = Trace\{ X\rightarrow \overset{g}{R} (X,Y)Z \}.
	\end{equation*}

	\section{Almost Hermitian manifolds}
	
	
	Depending on the properties of tensor $A$, we can observe various examples of the generalized Riemannian manifold (see \cite{ivanov2016}). An \textit{almost Hermitian manifold} $(\mathcal{M}, g, A)$ is an $n$-dimensional Riemannian manifold $(\mathcal{M}, g)$ (where $n=2k\geq 4$) equipped with almost complex structure $A$ which satisfies
	\begin{equation}\label{eq:conditionsforAHM}
		A^2 = - I,\ g(AX,AY)=g(X,Y).
	\end{equation}
	The fundamental 2-form $F$ (the so-called the \textit{K\"ahler form}) is defined by $F(X,Y)=g(AX,Y)$. The following equations also apply to the almost Hermitian manifold
	\begin{equation}\label{eq:FAXY}
		F(AX,Y)= -F(X,AY) = -g(X,Y) \;\;\; \mbox{and} \;\;\; F(AX,AY)=F(X,Y).
	\end{equation}
	From equations (\ref{eq:FgA}), (\ref{eq:conditionsforAHM}), (\ref{eq:FAXY}), we conclude that
	\begin{equation}\label{eq:GAXY}
		G(AX,Y)=-G(X,AY)=-G(Y,X)  \;\;\; \mbox{and} \;\;\; G(AX,AY)=G(X,Y).
	\end{equation}
	
	An almost Hermitian manifold $(\mathcal{M}, g, A)$ can be considered as a generalized Riemannian manifold $(\mathcal{M}, G=g+F)$ that satisfies relations (\ref{eq:conditionsforAHM}) (see \cite{ivanov2016,prvanovic1995}), where the skew-symmetric part $F$ of basic tensor $G$ is defined with $F(X,Y)=g(AX,Y)$. We will observe an almost Hermitian manifold $(\mathcal{M}, g, A)$ with a quarter-symmetric $G$-metric connection (\ref{eq:Q-S-Gmetric}).  Actually, such a connection preserves the almost Hermitian structure $(g, A)$, i.e. $\overset{1}{\nabla} g =\overset{1}{\nabla} A =0$. A linear connection that preserves the almost Hermitian structure is called the \textit{almost Hermitian connection} (also known as a \textit{natural connection}) (for instance, see \cite{gauduchon1997,yu2016}), which means that quarter-symmetric $G$-metric connection is an almost Hermitian connection.
	
	\begin{theorem}
		The torsion tensor $\overset{1}{T}$ of quarter-symmetric connection (\ref{eq:Q-S-Gmetric}) on an almost Hermitian manifold satisfies the following relations
		\begin{align*} 
			A\overset{1}{T}(AX,AY) & =A\overset{1}{T}(X,Y) - \overset{1}{T}(AX,Y) - \overset{1}{T}(X,AY),	
			\\ 
			\overset{1}{T}(X,Y,Z) & = \overset{1}{T}(AX,AY,Z) + \overset{1}{T}(AX,Y,AZ) + \overset{1}{T}(X,AY,AZ),
			\\ 
			\underset{XYZ}{\sigma} \overset{1}{T}(X,Y,Z) & = \underset{XYZ}{\sigma} (\overset{1}{T}(AX,Y,AZ) + \overset{1}{T}(X,AY,AZ)), 	
		\end{align*}
		where $\underset{XYZ}{\sigma}$ denote the cyclic sum with respect to the vector fields $X,Y,Z$.	
	\end{theorem}
	\begin{proof} 
		These relations can be proven by using the property of skew-symmetric 2-form $F$ in almost Herimitian manifolds, i.e. by using equations (\ref{eq:conditionsforAHM}) and (\ref{eq:FAXY}). 
		
	\end{proof}
	
	An almost Hermitian manifold is a \textit{K\"ahler manifold} if $\overset{g}{\nabla} A =0$. From equation (\ref{eq:covderivofA-QSGMC}), we see that structure tensor $A$ is parallel with respect to the Levi-Civita connection, and it implies the following statement.
	
	\begin{theorem}\label{thm:AHMwithQ-S-GmetricisKn}
		The almost Hermitian manifold $(\mathcal{M}, g, A)$ with quarter-symmetric connection (\ref{eq:Q-S-Gmetric}) preserving the generalized Riemannian metric $G$ is the K\"ahler manifold.
	\end{theorem}
	
	Following the previous theorem, further consideration will be related with the K\"ahler manifold. For this manifold, the term "generalized metric (i.e. $G$-metric) connection" is equivalent to the term "metric $A$-connection" (note that many papers use the term "metric $F$-connection", as $F$ is used to denote (1,1) structure tensor).
	
	The Riemannian curvature tensor $\overset{g}{R}$ on the K\"ahler manifold $(\mathcal{M}, g, A)$ satisfies the following relations (for instance see \cite{chaturvedi2015,mikes2015})
	\begin{align}
		\label{eq:relation1forRg}
		\overset{g}{R} (X,Y)AZ & = A\overset{g}{R} (X,Y)Z, \\
		\label{eq:relation2forRg}
		\overset{g}{R} (X,Y,AZ,AW) & = \overset{g}{R} (AX,AY,Z,W), \\
		\label{eq:relation3forRg}
		\overset{g}{R} (X,AY,AZ,W) & = \overset{g}{R} (AX,Y,Z,AW), \\
		\label{eq:relation4forRg}
		\overset{g}{R} (AX,AY,AZ,AW) & = \overset{g}{R} (X,Y,Z,W), \\
		\label{eq:relation5forRg}
		\overset{g}{R} (X,Y,Z,AW) & = - \overset{g}{R} (X,Y,AZ,W).
	\end{align}

	Moreover, if the (0,2) type tensor $B$ is \textit{hybrid}, then it holds (for instance see \cite{ozdemiryildirim2005}, \cite{vishnevskii1985} pp. 31)
	\begin{equation*}
		B(AX,Y)=-B(X,AY).
	\end{equation*}
	On the K\"ahler manifold, we have
	\begin{equation*}
		B(AX,AY)=B(X,Y).
	\end{equation*}
	For example, on the K\"ahler manifold, tensors $g$ and $F$ are hybrid (see \cite{yano1965} pp. 192), from which it follows that generalized Riemannian metric $G$ is hybrid (this is shown by equations (\ref{eq:conditionsforAHM}), (\ref{eq:FAXY}) and (\ref{eq:GAXY})). Also, on the K\"ahler manifold, the Ricci tensor $\overset{g}{R}ic$ is a hybrid tensor (see \cite{yano1965} pp. 68), i.e. satisfies relation 
	\begin{equation}\label{eq:Riccihybrid}
		\overset{g}{R}ic(X,AY) = - \overset{g}{R}ic(AX,Y).
	\end{equation}
	
	In the following theorem, we state the results we will use to study curvature tensor properties.
	
	\begin{theorem} Let $(\mathcal{M}, g, A)$ be a K\"ahler manifold.
		\begin{itemize}
			\item[(1)] If $\overset{g}{\nabla} \pi$ is hybrid tensor, then $\overset{1}{D}$ is also hybrid, i.e. it holds that
			\begin{equation}\label{eq:Dhybrid}
				\overset{1}{D}(AX,Y)= - \overset{1}{D}(X,AY), \; \overset{1}{D}(AX,AY)= \overset{1}{D}(X,Y),
			\end{equation}
			where $\overset{1}{D}$ is given by (\ref{eq:D1}).
			\item[(2)] If $\overset{g}{\nabla} \pi$ and $\pi\otimes\pi$ are hybrid tensors, then $\overset{\theta}{D}$ are hybrid, $\theta=0,2,3$, i.e. it holds that
			\begin{equation}\label{eq:thetaDhybrid}
				\overset{\theta}{D}(AX,Y)= - \overset{\theta}{D}(X,AY), \; \overset{\theta}{D}(AX,AY)= \overset{\theta}{D}(X,Y), \; \theta=0,2,3,
			\end{equation}
			where $\overset{\theta}{D}$ are given by (\ref{eq:D0}), (\ref{eq:D2}), (\ref{eq:D3}).
		\end{itemize}
	\end{theorem}
	\begin{proof}
		We will prove equations (\ref{eq:thetaDhybrid}) for $\theta=2$.
		From equation (\ref{eq:D2}), on the K\"ahler manifold, we have
		\begin{equation}\label{eq:D2AXY+D2XAY}
			\overset{2}{D}(AX,Y) + \overset{2}{D}(X,AY) = (\overset{g}{\nabla}_{AX} \pi)(Y) + (\overset{g}{\nabla}_{X} \pi)(AY) + \pi(AX)\pi(AY) - \pi(X)\pi(Y).
		\end{equation}
		If $\overset{g}{\nabla} \pi$ and $\pi\otimes\pi$ are hybrid tensors, then it holds
		\begin{equation}\label{eq:hybrid1}
			(\overset{g}{\nabla}_{AX} \pi)(Y) =- (\overset{g}{\nabla}_{X} \pi)(AY), \; \pi(AX)\pi(Y) = - \pi(X)\pi(AY) 
		\end{equation}
		and 
		\begin{equation}\label{eq:hybrid2}
			(\overset{g}{\nabla}_{AX} \pi)(AY) = (\overset{g}{\nabla}_{X} \pi)(Y), \; \pi(AX)\pi(AY) = \pi(X)\pi(Y). 
		\end{equation}
		
		Applying equations (\ref{eq:hybrid1}) and (\ref{eq:hybrid2}) to (\ref{eq:D2AXY+D2XAY}), we get
		\begin{equation*}
			\overset{2}{D}(AX,Y) + \overset{2}{D}(X,AY) =0.
		\end{equation*}
		If we replace $X$ with $AX$ in the previous equation and using $A^2=-I$, we obtain the second equation of (\ref{eq:thetaDhybrid}).
	\end{proof}

	\begin{remark}
		Theorem \ref{thm:AHMwithQ-S-GmetricisKn} is the equivalent form of Theorem 4.2 in \cite{yano1982a}: In order that the covariant derivative of the complex structure tensor of a Hermitian manifold with respect to the quarter-symmetric metric connection vanish, it is necessary and sufficient that the Hermitian manifold be a K\"ahler manifold.
	\end{remark}
	
	\section{Curvature properties of quarter-symmetric connection on K\"ahler manifold}
	
	In this section, we will consider the properties of the curvature tensor on the K\"ahler manifold with a quarter-symmetric connection (\ref{eq:Q-S-Gmetric}). The papers \cite{bhowmik2010, dubey2010, mishrapandey1980, yano1982, rastogi1986} studied the quarter-symmetric connection on the K\"ahler manifold. For example, in paper \cite{yano1982}, K. Yano and T. Imai proved that the K\"ahler manifold with a quarter-symmetric  metric $A$-connection (\ref{eq:Q-S-Gmetric}) is flat if curvature tensor $\overset{1}{R}$ vanishes. 
	
	Using the linearly independent curvature tensors with respect to quarter-symmetric connection (\ref{eq:Q-S-Gmetric}), below we will construct the tensors that are independent of the choice of quarter-symmetric connection generator.
	
	\subsection{Curvature tensor of the first kind}
	
	The curvature tensor of the first kind on the K\"ahler manifold with a quarter-symmetric metric $A$-connection (\ref{eq:Q-S-Gmetric}) is given by equation
	\begin{equation}\label{eq:R1XYZAHM}
		\overset{1}{R} (X,Y)Z = \overset{g}{R} (X,Y)Z  -  \overset{1}{D}(X,Y) A Z,   
	\end{equation}
	where $\overset{1}{D}$ is tensor given by (\ref{eq:D1}). By contracting with respect to vector field $X$ in equation (\ref{eq:R1XYZAHM}), we obtain
	\begin{equation*}
		\overset{1}{R}ic (Y,Z) = \overset{g}{R}ic(Y,Z)  -  \overset{1}{D}(AZ,Y). 
	\end{equation*}
	If we replace $Z$ with $AZ$ in the previous equation, we have
	\begin{equation*}
		\overset{1}{R}ic (Y,AZ) = \overset{g}{R}ic(Y,AZ)  -  \overset{1}{D}(A^2 Z,Y), 
	\end{equation*}
	from which we obtain
	\begin{equation}\label{eq:D1Ric1}
		\overset{1}{D}(Z,Y) =
		\overset{1}{R}ic (Y,AZ) - \overset{g}{R}ic(Y,AZ). 
	\end{equation}
	By substituting (\ref{eq:D1Ric1}) into (\ref{eq:R1XYZAHM}), we obtain
	\begin{equation*}
		\overset{1}{R} (X,Y)Z = \overset{g}{R} (X,Y)Z  -  (  \overset{1}{R}ic (Y,AX) - \overset{g}{R}ic(Y,AX)   ) A Z.   
	\end{equation*}
	By separating the elements of connections $\overset{1}{\nabla}$ and $\overset{g}{\nabla}$, we get the relation
	\begin{equation}\label{eq:R1XYZgmAHM2}
		\overset{1}{R} (X,Y)Z + \overset{1}{R}ic (Y,AX) AZ = \overset{g}{R} (X,Y)Z  + \overset{g}{R}ic(Y,AX)  A Z   
	\end{equation}
	and based on that, we will formulate the following theorem.
	\begin{theorem}
		Let $(\mathcal{M}, g, A)$ be a K\"ahler manifold with a quarter-symmetric metric $A$-connection (\ref{eq:Q-S-Gmetric}). Tensor
		\begin{equation}\label{eq:K1XYZgmAHM}
			\overset{1}{H} (X,Y)Z = \overset{1}{R} (X,Y)Z + \overset{1}{R}ic (Y,AX) AZ 
		\end{equation}
		is independent of generator $\pi$.
	\end{theorem}
	
	In this part, we will also deal with some other properties of the curvature tensors on the K\"ahler manifold, depending on the quarter-symmetric connection generator properties. We now state the properties of curvature tensors of the first kind.
	
	\begin{theorem}
		Let $(\mathcal{M}, g, A)$ be a K\"ahler manifold with a quarter-symmetric metric $A$-connection (\ref{eq:Q-S-Gmetric}).
		\begin{itemize}
			\item[(1)] If $\overset{g}{\nabla} \pi$ is hybrid, then the curvature tensor of the first kind and structure tensor $A$ satisfies the following relations
			\begin{align*}
				\overset{1}{R} (X,Y,AZ,AW) & = \overset{1}{R} (AX,AY,Z,W), \\
				\overset{1}{R} (X,AY,AZ,W) & = \overset{1}{R} (AX,Y,Z,AW), \\
				\overset{1}{R} (AX,AY,AZ,AW) & = \overset{1}{R} (X,Y,Z,W). 
			\end{align*}
			
			\item[(2)] The curvature tensor of the first kind and structure tensor $A$ satisfies the following relations
			\begin{align*}
				\overset{1}{R} (X,Y)AZ & = A\overset{1}{R} (X,Y)Z, \\
				\overset{1}{R} (X,Y,Z,AW) & = - \overset{1}{R} (X,Y,AZ,W).
			\end{align*}
		\end{itemize}
	\end{theorem}
	\begin{proof}
		From equation (\ref{eq:R1XYZgm}), we obtain the (0,4) type curvature tensor of the first kind
		\begin{equation*}
			\overset{1}{R} (X,Y,Z,W) = \overset{g}{R} (X,Y,Z,W)  -  \overset{1}{D}(X,Y) F(Z,W).   
		\end{equation*}
		From here, we have
		\begin{equation}\label{eq:R1XYAZAWAHM}
			\begin{split}
				\overset{1}{R} (X,Y,AZ,AW) & = \overset{g}{R} (X,Y,AZ,AW)  -  \overset{1}{D}(X,Y) F(AZ,AW) 
				\\ &  =   \overset{g}{R} (X,Y,AZ,AW)  -  \overset{1}{D}(X,Y) F(Z,W), 
			\end{split}
		\end{equation}
		where we used equation (\ref{eq:FAXY}). On the other hand, we have
		\begin{equation}\label{eq:R1AXAYZWAHM}
			\overset{1}{R} (AX,AY,Z,W) = \overset{g}{R} (AX,AY,Z,W)  -  \overset{1}{D}(AX,AY) F(Z,W).   
		\end{equation}
		After subtracting equation (\ref{eq:R1AXAYZWAHM}) from (\ref{eq:R1XYAZAWAHM}) and using (\ref{eq:relation2forRg}), we get
		\begin{equation*}
			\overset{1}{R} (X,Y,AZ,AW) - \overset{1}{R} (AX,AY,Z,W) = (\overset{1}{D}(AX,AY) - \overset{1}{D}(X,Y))F(Z,W).
		\end{equation*}
		From equation (\ref{eq:Dhybrid}), we see that
		\begin{equation*}
			\overset{1}{R} (X,Y,AZ,AW) = \overset{1}{R} (AX,AY,Z,W)
		\end{equation*}
		if $\overset{g}{\nabla} \pi$ is hybrid. Other relations are proved analogously.
	\end{proof}
	
	
	\subsection{Curvature tensor of the second kind}
	
	Using the curvature tensor of the second kind, we can get a new tensor on the K\"ahler manifold that is independent of quarter-symmetric connection generator $\pi$.
	\begin{theorem}
		Let $(\mathcal{M}, g, A)$ be a K\"ahler manifold with a quarter-symmetric metric $A$-connection (\ref{eq:Q-S-Gmetric}). Tensor
		\begin{equation}\label{eq:K2XYZgmAHM}
			\overset{2}{H} (X,Y)Z =  \overset{2}{R} (X,Y)Z  + \overset{2}{R}ic(AX,Z) A Y - \overset{2}{R}ic(AY,Z) AX
		\end{equation}
		is independent of generator $\pi$.
	\end{theorem}
	\begin{proof}
		The curvature tensor of the second kind with respect to quarter-symmetric connection (\ref{eq:Q-S-Gmetric}) reads
		\begin{equation}\label{eq:R2XYZgm2}
			\overset{2}{R} (X,Y)Z = \overset{g}{R} (X,Y)Z  -  \overset{2}{D} (X,Z)A Y +  \overset{2}{D} (Y,Z)AX,
		\end{equation}
		where $\overset{2}{D}$ is $(0,2)$ type tensor given by (\ref{eq:D2}). By contracting vector field $X$ in equation (\ref{eq:R2XYZgm2}), we have
		\begin{equation}\label{eq:Ric2gmAHM}
			\overset{2}{R}ic (Y,Z) = \overset{g}{R}ic(Y,Z) -  \overset{2}{D} (AY,Z),
		\end{equation}
		where we used that the structure tensor $A$ is trace-free, i.e. $Trace\{X\rightarrow AX\}=0$. From equation (\ref{eq:Ric2gmAHM}), we have
		\begin{equation*}
			\overset{2}{D} (A^2Y,Z) =  \overset{g}{R}ic(AY,Z) - \overset{2}{R}ic (AY,Z)
		\end{equation*}
		and further 
		\begin{equation}\label{eq:D22}
			\overset{2}{D} (Y,Z) =  \overset{2}{R}ic(AY,Z) - \overset{g}{R}ic (AY,Z).
		\end{equation}
		By combining equations (\ref{eq:R2XYZgm2}) and (\ref{eq:D22}), we find
		\begin{equation*}
			\overset{2}{R} (X,Y)Z = \overset{g}{R} (X,Y)Z  -  (  \overset{2}{R}ic(AX,Z) - \overset{g}{R}ic (AX,Z)   )  A Y +   ( \overset{2}{R}ic(AY,Z) - \overset{g}{R}ic (AY,Z))AX,
		\end{equation*}
		from which
		\begin{equation}\label{eq:R2XYZgmAHM3}
			\overset{2}{R} (X,Y)Z  + \overset{2}{R}ic(AX,Z) A Y - \overset{2}{R}ic(AY,Z) AX = \overset{g}{R} (X,Y)Z  + \overset{g}{R}ic (AX,Z) A Y  - \overset{g}{R}ic (AY,Z) AX.
		\end{equation}
	\end{proof}
	
	Depending on generator $\pi$ property, the curvature tensor of the second kind and structure tensor $A$ have the following properties.
	\begin{theorem}
		Let $(\mathcal{M}, g, A)$ be a K\"ahler manifold with a quarter-symmetric metric $A$-connection (\ref{eq:Q-S-Gmetric}).
		\begin{itemize}
			\item[(1)] If $\overset{g}{\nabla} \pi$ and $\pi\otimes\pi$ are hybrid, then the curvature tensor of the second kind and structure tensor $A$ satisfies the following relations
			\begin{align*}
				\overset{2}{R} (X,Y,AZ,AW) & = \overset{2}{R} (AX,AY,Z,W), \\
				\overset{2}{R} (X,AY,AZ,W) & = \overset{2}{R} (AX,Y,Z,AW), \\
				\overset{2}{R} (AX,AY,AZ,AW) & = \overset{2}{R} (X,Y,Z,W). 
			\end{align*}
			
			\item[(2)] The curvature tensor of the second kind and  the structure tensor $A$ satisfies the following relations
			\begin{align*}
				\overset{2}{R} (X,Y)AZ & = A\overset{2}{R} (X,Y)Z, \\
				\overset{2}{R} (X,Y,Z,AW) & = - \overset{2}{R} (X,Y,AZ,W),
			\end{align*}
			if and only if
			\begin{equation*}
				\overset{2}{D} (X,Z)Y + 	\overset{2}{D} (X,AZ)AY= \overset{2}{D} (Y,Z)X + 	\overset{2}{D} (Y,AZ)AX,
			\end{equation*}
			where $\overset{2}{D}$ given with (\ref{eq:D2}).
		\end{itemize}
	\end{theorem}
	\begin{proof}
		The (0,4) type curvature tensor of the second kind is given by equation
		\begin{equation*}
			\overset{2}{R} (X,Y,Z,W) = \overset{g}{R} (X,Y,Z,W)  -  \overset{2}{D} (X,Z)F(Y,W) +  \overset{2}{D} (Y,Z)F(X,W),
		\end{equation*}
		from which it follows
		\begin{align*}
			\overset{2}{R} (X,AY,AZ,W) & = \overset{g}{R} (X,AY,AZ,W)  +  \overset{2}{D} (X,AZ)g(Y,W) +  \overset{2}{D} (AY,AZ)F(X,W), \\
			\overset{2}{R} (AX,Y,Z,AW) & = \overset{g}{R} (AX,Y,Z,AW)  -  \overset{2}{D} (AX,Z)g(Y,W) +  \overset{2}{D} (Y,Z)F(X,W),
		\end{align*}
		where we used relations (\ref{eq:FAXY}). By subtracting the previous two equations and using (\ref{eq:relation3forRg}), we get
		\begin{equation*}
			\begin{split}
				\overset{2}{R} (X,AY,AZ,W) - \overset{2}{R} (AX,Y,Z,AW)  = & (\overset{2}{D} (X,AZ) + \overset{2}{D} (AX,Z) )g(Y,W) 
				\\ & + (\overset{2}{D} (AY,AZ)- \overset{2}{D} (Y,Z))F(X,W).
			\end{split}
		\end{equation*}
		If $\overset{g}{\nabla} \pi$ and $\pi\otimes\pi$ are hybrid, then the relation (\ref{eq:thetaDhybrid}) holds, and we verified that
		\begin{equation*}
			\overset{2}{R} (X,AY,AZ,W) = \overset{2}{R} (AX,Y,Z,AW).
		\end{equation*}
	\end{proof}

	\subsection{Curvature tensor of the third kind}
	
	The curvature tensor of the third kind $\overset{3}{R}$ with respect to quarter-symmetric connection (\ref{eq:Q-S-Gmetric}) is given by equation
	\begin{equation}\label{eq:R3XYZgm2}
		\overset{3}{R} (X,Y)Z = \overset{g}{R} (X,Y)Z  -  \overset{2}{D} (X,Y)A Z +  \overset{3}{D} (Y,Z)AX,
	\end{equation}
	where $\overset{2}{D}$ and $\overset{3}{D}$ are $(0,2)$ type tensors given by (\ref{eq:D2}) and (\ref{eq:D3}), respectively.
	If we contract equation (\ref{eq:R3XYZgm2}) with respect to $X$, then we obtain the relation between Ricci tensors $\overset{3}{R}ic$ and $\overset{g}{R}ic$
	\begin{equation*}
		\overset{3}{R}ic (Y,Z) = \overset{g}{R}ic(Y,Z)  -  \overset{2}{D} (AZ,Y),
	\end{equation*}
	from which we get the following relation
	\begin{equation}\label{eq:D23}
		\overset{2}{D} (Z,Y) = \overset{3}{R}ic (Y,AZ) - \overset{g}{R}ic(Y,AZ).
	\end{equation}
	On the other hand, if we contract equation (\ref{eq:R3XYZgm2}) with respect to vector field $Z$, then we get
	\begin{equation}\label{eq:'R3AHM}
		\overset{3}{'R} (X,Y) =     \overset{3}{D} (Y,AX),
	\end{equation}
	where we used $Trace\{Z\rightarrow \overset{g}{R} (X,Y)Z\}=0$ and denoted $\overset{3}{'R}(X,Y) = Trace\{Z\rightarrow \overset{3}{R} (X,Y)Z\}$. Further, it follows that
	\begin{equation}\label{eq:D31}
		\overset{3}{D} (Y,X) = - \overset{3}{'R} (AX,Y),
	\end{equation}
	where we take into account that $A^2=-I$. By replacing equations (\ref{eq:D23}) and (\ref{eq:D31}) into (\ref{eq:R3XYZgm2}), we have
	\begin{equation*}
		\overset{3}{R} (X,Y)Z = \overset{g}{R} (X,Y)Z  -  (\overset{3}{R}ic (Y,AX) - \overset{g}{R}ic(Y,AX)) AZ   - \overset{3}{'R} (AZ,Y) AX,
	\end{equation*}
	and further
	\begin{equation}\label{eq:R3XYZgm4}
		\overset{3}{R} (X,Y)Z + \overset{3}{R}ic (Y,AX) AZ + \overset{3}{'R} (AZ,Y) AX = \overset{g}{R} (X,Y)Z  + \overset{g}{R}ic(Y,AX) AZ.
	\end{equation}
	Finally, we have proved the following theorem.
	\begin{theorem}
		Let $(\mathcal{M}, g, A)$ be a K\"ahler manifold with a quarter-symmetric metric $A$-connection (\ref{eq:Q-S-Gmetric}). Tensor
		\begin{equation}\label{eq:K3XYZgmAHM}
			\overset{3}{H} (X,Y)Z =  \overset{3}{R} (X,Y)Z + \overset{3}{R}ic (Y,AX) AZ + \overset{3}{'R} (AZ,Y) AX
		\end{equation}
		is independent of generator $\pi$.
	\end{theorem}
	
	By comparing equations (\ref{eq:R1XYZgmAHM2}) and (\ref{eq:R3XYZgm4}), we conclude that
	\begin{equation*}
		\overset{1}{H} (X,Y)Z = \overset{3}{H} (X,Y)Z.
	\end{equation*}
	Based on expressions for tensor $\overset{2}{D}$, i.e. from equations (\ref{eq:D22}) and (\ref{eq:D23}), it follows that
	\begin{equation*}
		\overset{2}{R}ic (X, Y) = \overset{3}{R}ic (Y,X).
	\end{equation*}
	
	In the following statement, we state the properties of the curvature tensor of the third kind, which can be proved similarly to the properties of the previous tensors.
	
	\begin{theorem}
		Let $(\mathcal{M}, g, A)$ be a K\"ahler manifold with a quarter-symmetric metric $A$-connection (\ref{eq:Q-S-Gmetric}).
		\begin{itemize}
			\item[(1)] If $\overset{g}{\nabla} \pi$ and $\pi\otimes\pi$ are hybrid, then the curvature tensor of the third kind and structure tensor $A$ satisfies the following relations
			\begin{align*}
				\overset{3}{R} (X,Y,AZ,AW) & = \overset{3}{R} (AX,AY,Z,W), \\
				\overset{3}{R} (X,AY,AZ,W) & = \overset{3}{R} (AX,Y,Z,AW), \\
				\overset{3}{R} (AX,AY,AZ,AW) & = \overset{3}{R} (X,Y,Z,W). 
			\end{align*}
			
			\item[(2)] The curvature tensor of the third kind and structure tensor $A$ satisfies the following relations
			\begin{align*}
				\overset{3}{R} (X,Y)AZ & = A\overset{3}{R} (X,Y)Z, \\
				\overset{3}{R} (X,Y,Z,AW) & = - \overset{3}{R} (X,Y,AZ,W),
			\end{align*}
			if and only if 
			\begin{equation*}
				\overset{3}{D} (Y,Z)X = - \overset{3}{D} (Y,AZ)AX,
			\end{equation*}
			where $\overset{3}{D}$ given by (\ref{eq:D3}).
		\end{itemize}
	\end{theorem}
	
	\subsection{Curvature tensor of the fourth kind}
	
	The equation of the curvature tensor of the fourth kind $\overset{4}{R} $ on the K\"ahler manifold with a quarter-symmetric connection (\ref{eq:Q-S-Gmetric}) take the form
	\begin{equation}\label{eq:R4XYZgmAHM}
		\overset{4}{R} (X,Y)Z =  \overset{g}{R} (X,Y)Z  -  \overset{3}{D} (X,Y)A Z +  \overset{3}{D} (Y,Z)AX + \pi (Z) (\pi (Y) X - \pi (X) Y ),
	\end{equation}
	where $\overset{3}{D}$ is given by equation (\ref{eq:D3}). If we contract with respect to vector  $X$  in equation (\ref{eq:R4XYZgmAHM}), then we obtain the relation between the Ricci tensor of the fourth kind and the Ricci tensor of metric $g$ 
	\begin{equation}\label{eq:Ric4}
		\overset{4}{R}ic (Y,Z) =  \overset{g}{R}ic (Y,Z)  -  \overset{3}{D} (AZ,Y)  + (n-1)\pi (Y)\pi (Z).
	\end{equation}
	On the other hand, by contracting equation (\ref{eq:R4XYZgmAHM}) with respect to $Z$, we have the following equation
	\begin{equation}\label{eq:'R4AHM}
		\overset{4}{'R} (X,Y) =  \overset{3}{D} (Y,AX),
	\end{equation}
	from which we obtain
	\begin{equation}\label{eq:D34}
		\overset{3}{D} (Y,X) = -\overset{4}{'R} (AX,Y),
	\end{equation}
	where $\overset{4}{'R}(X,Y) = Trace\{Z\rightarrow \overset{4}{R} (X,Y)Z\}$. From (\ref{eq:Ric4}) and (\ref{eq:D34}), we have
	\begin{equation}\label{eq:piYpiZ}
		\pi (Y)\pi (Z) = \frac{1}{n-1} ( \overset{4}{R}ic (Y,Z) -  \overset{g}{R}ic (Y,Z) -  \overset{4}{'R} (AY,AZ) ).
	\end{equation}
	By substituting equations (\ref{eq:D34}) and (\ref{eq:piYpiZ}) into (\ref{eq:R4XYZgmAHM}), after simple rearranging, we obtain
	\begin{equation*}
		\begin{split}
			\overset{4}{R} (X,Y)Z  & - \overset{4}{'R} (AY,X)AZ + \overset{4}{'R} (AZ,Y)AX 
			\\ & - \frac{1}{n-1} ( \overset{4}{R}ic (Y,Z)X - \overset{4}{R}ic (X,Z)Y -  \overset{4}{'R} (AY,AZ)X + \overset{4}{'R} (AX,AZ)Y  )
			\\ & =  \overset{g}{R} (X,Y)Z +  \frac{1}{n-1}  (  \overset{g}{R}ic (X,Z)Y - \overset{g}{R}ic (Y,Z)X ) = \overset{g}{W} (X,Y)Z,
		\end{split}
	\end{equation*}
	where $\overset{g}{W}$ is the Weyl projective curvature tensor (\ref{eq:Weylptensor}). The tensor of the left-hand side of the previous equation is independent of the choice of a 1-form $\pi$.
	
	\begin{theorem}\label{thm:H4KM}
		Let $(\mathcal{M}, g, A)$ be a K\"ahler manifold with a quarter-symmetric metric $A$-connection (\ref{eq:Q-S-Gmetric}). Tensor
		\begin{equation}\label{eq:K4XYZgmAHM}
			\begin{split}
				\overset{4}{H} (X,Y)Z = & \overset{4}{R} (X,Y)Z   - \overset{4}{'R} (AY,X)AZ + \overset{4}{'R} (AZ,Y)AX 
				\\ & - \frac{1}{n-1} ( \overset{4}{R}ic (Y,Z)X - \overset{4}{R}ic (X,Z)Y -  \overset{4}{'R} (AY,AZ)X + \overset{4}{'R} (AX,AZ)Y  ) 
			\end{split}
		\end{equation}
		is independent of generator $\pi$ and it is equal to the Weyl projective curvature tensor $\overset{g}{W}$.
	\end{theorem}
	Immediately, we have the following corollary.
	\begin{corollary}
		Let $(\mathcal{M}, g, A)$ be a K\"ahler manifold with a quarter-symmetric metric $A$-connection (\ref{eq:Q-S-Gmetric}). If Ricci tensor $\overset{4}{R}ic$ and tensor $\overset{4}{'R}$ vanish on this manifold, then the curvature tensor of the fourth kind and the Weyl projective curvature tensor are equal, i.e. $\overset{4}{R} = \overset{g}{W}$.
	\end{corollary}
	
	From equations (\ref{eq:'R3AHM}) and (\ref{eq:'R4AHM}), we obtain the relation
	\begin{equation*}
		\overset{3}{'R} = \overset{4}{'R}.
	\end{equation*}
	
	Using relations (\ref{eq:relation1forRg})-(\ref{eq:relation5forRg}), we can easily prove some relations for the curvature tensor of the fourth kind.
	
	\begin{theorem}
		Let $(\mathcal{M}, g, A)$ be a K\"ahler manifold with a quarter-symmetric metric $A$-connection (\ref{eq:Q-S-Gmetric}).
		\begin{itemize}
			\item[(1)] If $\overset{g}{\nabla} \pi$ and $\pi\otimes\pi$ are hybrid, then the curvature tensor of the fourth kind and structure tensor $A$ satisfies the following relations
			\begin{align*}
				\overset{4}{R} (X,Y,AZ,AW) & = \overset{4}{R} (AX,AY,Z,W), \\
				\overset{4}{R} (X,AY,AZ,W) & = \overset{4}{R} (AX,Y,Z,AW), \\
				\overset{4}{R} (AX,AY,AZ,AW) & = \overset{4}{R} (X,Y,Z,W). 
			\end{align*}
			
			\item[(2)] The curvature tensor of the fourth kind and structure tensor $A$ satisfies the following relations
			\begin{align*}
				\overset{4}{R} (X,Y)AZ & = A\overset{4}{R} (X,Y)Z, \\
				\overset{4}{R} (X,Y,Z,AW) & = - \overset{4}{R} (X,Y,AZ,W),
			\end{align*}
			if and only if
			\begin{equation*}
				\overset{4}{D} (Y,Z)AX + \pi(X) \pi(Z)AY = - \overset{4}{D} (Y,AZ)X + \pi(X) \pi(AZ)Y,
			\end{equation*}
			where $\overset{4}{D}(Y,Z)=(\overset{g}{\nabla}_{Y} \pi)(AZ) - 2\pi(Y) \pi(Z)$. 
			
		\end{itemize}
	\end{theorem}
	
	
	\subsection{Curvature tensor of the fifth kind}
	
	We will prove the following theorem using the curvature tensor of the fifth kind.
	
	\begin{theorem}
		Let $(\mathcal{M}, g, A)$ be a K\"ahler manifold with a quarter-symmetric metric $A$-connection (\ref{eq:Q-S-Gmetric}). Tensor
		\begin{equation}\label{eq:K5XYZgmAHM}
			\begin{split}
				\overset{5}{H} (X,Y)Z  =  \overset{5}{R} (X,Y)Z  & + \frac{1}{n-1} ( \overset{5}{R}ic (X,Y)Z - \overset{5}{R}ic (Y,Z)X )
				- \frac{1}{2(n-1)} ( \overset{1}{R}ic (X,Y)Z  - \overset{1}{R}ic (Y,Z)X )
				\\ & - \frac{1}{2(n-1)} (   \overset{3}{'R} (AY,AX)Z - \overset{3}{'R} (AZ,AY)X  ) 
				\\ & 
				+ \frac{1}{2} ( \overset{1}{R}ic (Y,AX)AZ - \overset{3}{R}ic (Z,AY)AX -  \overset{3}{'R} (AZ,X)AY  )
			\end{split}
		\end{equation}
		is independent of generator $\pi$.
	\end{theorem}
	\begin{proof}
		If we take into account that
		\begin{equation*}
			\overset{1}{D}(X,Y) =  \overset{2}{D} (X,Y) - \overset{3}{D} (Y,X),
		\end{equation*}
		where $\overset{1}{D}$, $\overset{2}{D}$, $\overset{3}{D}$ are given by (\ref{eq:D1}), (\ref{eq:D2}), (\ref{eq:D3}), respectively, then the curvature tensor of the fifth kind on the K\"ahler manifold with a quarter-symmetric metric $A$-connection (\ref{eq:Q-S-Gmetric}) takes the following form
		\begin{equation}
			\label{eq:R5XYZgmAHM}
			\begin{split}
				\overset{5}{R} (X,Y)Z =  \overset{g}{R} (X,Y)Z - \frac{1}{2}  \overset{1}{D}(X,Y) A Z    -  \frac{1}{2} \overset{3}{D} (X,Z) A Y + \frac{1}{2} \overset{2}{D} (Y,Z) A X - \frac{1}{2} \pi (Y) (\pi (X) Z - \pi (Z) X  ).
			\end{split}
		\end{equation}
		By contracting with respect to vector field $X$ in the previous equation gives
		\begin{equation*}
			\begin{split}
				\overset{5}{R}ic (Y,Z) = & \overset{g}{R}ic (Y,Z) - \frac{1}{2}\overset{1}{D} (AZ,Y)   -  \frac{1}{2} \overset{3}{D} (AY,Z) + \frac{n-1}{2} \pi (Y) \pi (Z).
			\end{split}
		\end{equation*}
		From here, by using equations (\ref{eq:D1Ric1}) and (\ref{eq:D31}), it follows that
		\begin{equation}
			\label{eq:piYpiZ5}
			\begin{split}
				\pi (Y) \pi (Z) =  \frac{1}{n-1} ( 	2\overset{5}{R}ic (Y,Z) -   \overset{1}{R}ic (Y,Z)  -  \overset{3}{'R} (AZ,AY)  -  \overset{g}{R}ic (Y,Z)).
			\end{split}
		\end{equation}
		By substituting equations (\ref{eq:D1Ric1}), (\ref{eq:D23}), (\ref{eq:D31}) and (\ref{eq:piYpiZ5}) into (\ref{eq:R5XYZgmAHM}), after rearranging, we obtain
		\begin{equation}
			\label{eq:R5XYZgmAHM2}
			\begin{split}
				\overset{5}{R} (X,Y)Z  & + \frac{1}{n-1} ( \overset{5}{R}ic (X,Y)Z - \overset{5}{R}ic (Y,Z)X )
				- \frac{1}{2(n-1)} ( \overset{1}{R}ic (X,Y)Z  - \overset{1}{R}ic (Y,Z)X )
				\\ & - \frac{1}{2(n-1)} (   \overset{3}{'R} (AY,AX)Z - \overset{3}{'R} (AZ,AY)X  ) 
				\\ & 
				+ \frac{1}{2} ( \overset{1}{R}ic (Y,AX)AZ - \overset{3}{R}ic (Z,AY)AX -  \overset{3}{'R} (AZ,X)AY  )
				\\ & =  \overset{g}{R} (X,Y)Z  + \frac{1}{2(n-1)} ( \overset{g}{R}ic (X,Y)Z - \overset{g}{R}ic (Y,Z)X )
				+ \frac{1}{2} ( \overset{g}{R}ic (AX,Y)AZ - \overset{g}{R}ic (AY,Z)AX)
			\end{split}
		\end{equation}
		and thereby, we proved the theorem.
	\end{proof}
	
	Analogously, we can prove the following theorem.
	
	\begin{theorem}
		Let $(\mathcal{M}, g, A)$ be a K\"ahler manifold with a quarter-symmetric metric $A$-connection (\ref{eq:Q-S-Gmetric}).
		\begin{itemize}
			\item[(1)] If $\overset{g}{\nabla} \pi$ and $\pi\otimes\pi$ are hybrid, then the curvature tensor of the fifth kind and structure tensor $A$ satisfies the following relations
			\begin{align*}
				\overset{5}{R} (X,Y,AZ,AW) & = \overset{5}{R} (AX,AY,Z,W), \\
				\overset{5}{R} (X,AY,AZ,W) & = \overset{5}{R} (AX,Y,Z,AW), \\
				\overset{5}{R} (AX,AY,AZ,AW) & = \overset{5}{R} (X,Y,Z,W). 
			\end{align*}
			
			\item[(2)] The curvature tensor of the fifth kind and structure tensor $A$ satisfies the following relations
			\begin{align*}
				\overset{5}{R} (X,Y)AZ & = A\overset{5}{R} (X,Y)Z, \\
				\overset{5}{R} (X,Y,Z,AW) & = - \overset{5}{R} (X,Y,AZ,W).
			\end{align*}
			if and only if
			\begin{equation*}
				\overset{3}{D} (X,Z)Y + \overset{3}{D} (X,AZ)AY= (\overset{2}{D} (Y,Z)+ \pi(Y)\pi(AZ))X + 	(\overset{2}{D} (Y,AZ) - \pi(Y)\pi(Z))AX  ,
			\end{equation*}
			where $\overset{2}{D}$, $\overset{3}{D}$ given by (\ref{eq:D2}), (\ref{eq:D3}), respectively.
		\end{itemize}
	\end{theorem}
	
	\subsection{Curvature tensor of the zero kind}
	
	By the similar procedure as in the previous cases, using the curvature tensor of the zero kind, we can prove the following theorem.
	
	\begin{theorem}
		Let $(\mathcal{M}, g, A)$ be a K\"ahler manifold with a quarter-symmetric metric $A$-connection (\ref{eq:Q-S-Gmetric}). Tensor
		\begin{equation}\label{eq:K0XYZgmAHM}
			\begin{split}
				\overset{0}{H} (X,Y)Z  =  \overset{0}{R} (X,Y)Z  & + \frac{1}{n-1} ( \overset{0}{R}ic (X,Y)Z - \overset{0}{R}ic (Y,Z)X )
				- \frac{1}{2(n-1)} ( \overset{1}{R}ic (X,Z)Y  - \overset{1}{R}ic (Y,Z)X )
				\\ & - \frac{1}{4(n-1)} ( \overset{3}{R}ic (Z,X)Y  - \overset{3}{R}ic (Z,Y)X +  \overset{3}{'R} (AZ,AX)Y - \overset{3}{'R} (AZ,AY)X  ) 
				\\ & 
				+ \frac{1}{4} ( 2\overset{1}{R}ic (Y,AX)AZ + \overset{3}{R}ic (Z,AX)AY - \overset{3}{R}ic (Z,AY)AX ) 
				\\ &
				- \frac{1}{4} ( \overset{3}{'R} (AZ,X)AY - \overset{3}{'R} (AZ,Y)AX)
			\end{split}
		\end{equation}
		is independent of generator $\pi$.
	\end{theorem}
	\begin{proof}
		Based on equations (\ref{eq:D0}),(\ref{eq:D1}), (\ref{eq:D2}) and (\ref{eq:D3}), we have
		\begin{equation}\label{eq:relationsforDD0D2D3}
			\overset{1}{D}(X,Y) =  \overset{0}{D} (X,Y) - \overset{0}{D} (Y,X)  \quad \mbox{and} \quad 2\overset{0}{D} (X,Y) =  \overset{2}{D} (X,Y) +  \overset{3}{D} (X,Y).
		\end{equation}
		In view of equations (\ref{eq:R0XYZgm1}), (\ref{eq:conditionsforAHM}) and (\ref{eq:relationsforDD0D2D3}), the curvature tensor of the zero kind on the K\"ahler manifold with a quarter-symmetric metric $A$-connection (\ref{eq:Q-S-Gmetric}) takes the form
		\begin{equation}\label{eq:R0XYZgmAHM}
			\begin{split}
				\overset{0}{R} (X,Y)Z =  \overset{g}{R} (X,Y)Z & - \frac{1}{2} \overset{1}{D}(X,Y) A Z    -  \frac{1}{4} (\overset{2}{D} (X,Z) + \overset{3}{D} (X,Z)) A Y + \frac{1}{4} (\overset{2}{D} (Y,Z) + \overset{3}{D} (Y,Z)) A X 
				\\ & + \frac{1}{4} \pi (Z) (\pi (Y) X - \pi (X) Y  ),
			\end{split}
		\end{equation}
		where (0,2) type tensors $\overset{1}{D}$, $\overset{2}{D}$, $\overset{3}{D}$ are given by (\ref{eq:D1}), (\ref{eq:D2}), (\ref{eq:D3}), respectively. By contracting with respect to $X$ in the previous equation, we obtain
		\begin{equation}
			\label{eq:Ric0AHM}
			\begin{split}
				\overset{0}{R}ic (Y,Z) = & \overset{g}{R}ic (Y,Z)- \frac{1}{2}  {D} (AZ,Y)  -  \frac{1}{4} (\overset{2}{D} (AY,Z) + \overset{3}{D} (AY,Z)) + \frac{n-1}{4} \pi (Y) \pi (Z).
			\end{split}
		\end{equation}
		If we replace equations (\ref{eq:D1Ric1}), (\ref{eq:D23}), (\ref{eq:D31}) into (\ref{eq:Ric0AHM}), then we get
		\begin{equation}
			\label{eq:piYpiZ0}
			\begin{split}
				\pi (Y) \pi (Z) =  \frac{1}{n-1} ( 	4\overset{0}{R}ic (Y,Z) -   2\overset{1}{R}ic (Y,Z) - \overset{3}{R}ic (Z,Y) -  \overset{3}{'R} (AZ,AY)  -  \overset{g}{R}ic (Y,Z)).
			\end{split}
		\end{equation}
		Finally, by substituting (\ref{eq:D1Ric1}), (\ref{eq:D23}), (\ref{eq:D31}) and (\ref{eq:piYpiZ0}) into equation (\ref{eq:R0XYZgmAHM}), we obtain
		\begin{equation}\label{eq:H0Rg}
			\begin{split}
				\overset{0}{H} (X,Y)Z  =  \overset{g}{R} (X,Y)Z  & + \frac{1}{4(n-1)} ( \overset{g}{R}ic (X,Z)Y - \overset{g}{R}ic (Y,Z)X )
				\\ & + \frac{1}{4} ( 2\overset{g}{R}ic (AX,Y)AZ + \overset{g}{R}ic (AX,Z)AY - \overset{g}{R}ic (AY,Z)AX)
			\end{split}
		\end{equation}
		where $\overset{0}{H}$ is given by (\ref{eq:K0XYZgmAHM}).
	\end{proof}
	
	Now, we can give some other properties of the curvature tensor of the zero kind depending on generator $\pi$.
	
	\begin{theorem}
		Let $(\mathcal{M}, g, A)$ be a K\"ahler manifold with a quarter-symmetric metric $A$-connection (\ref{eq:Q-S-Gmetric}).
		\begin{itemize}
			\item[(1)] If $\overset{g}{\nabla} \pi$ and $\pi\otimes\pi$ are hybrid, then the curvature tensor of the zero kind and structure tensor $A$ satisfies the following relations
			\begin{align*}
				\overset{0}{R} (X,Y,AZ,AW) & = \overset{0}{R} (AX,AY,Z,W), \\
				\overset{0}{R} (X,AY,AZ,W) & = \overset{0}{R} (AX,Y,Z,AW), \\
				\overset{0}{R} (AX,AY,AZ,AW) & = \overset{0}{R} (X,Y,Z,W). 
			\end{align*}
			
			\item[(2)] If 
			\begin{equation*}
				(\overset{g}{\nabla}_{X} \pi)(Y) + \pi(X) \pi(AY) + \frac{1}{2} \pi(AX) \pi(Y) = 0,
			\end{equation*}
			then the curvature tensor of the zero kind and structure tensor $A$ satisfies the following relations
			\begin{align*}
				\overset{0}{R} (X,Y)AZ & = A\overset{0}{R} (X,Y)Z, \\
				\overset{0}{R} (X,Y,Z,AW) & = - \overset{0}{R} (X,Y,AZ,W).
			\end{align*}		
		\end{itemize}
	\end{theorem}
	\begin{proof}
		Equation (\ref{eq:R0XYZgm1}) implies the following
		\begin{align*}
			\begin{split}
				\overset{0}{R} (X,Y,Z,AW) = & \overset{g}{R} (X,Y,Z,AW) - \frac{1}{2}   ( \overset{0}{D} (X,Y) - \overset{0}{D} (Y,X)  )g(Z,W)    -  \frac{1}{2} \overset{0}{D} (X,Z) g(Y,W)
				\\ & + \frac{1}{2} \overset{0}{D} (Y,Z) g(X,W)  - \frac{1}{4} \pi(Z) (\pi (Y) F(X,W) - \pi (X) F(Y,W)),
			\end{split}
			\\
			\begin{split}
				\overset{0}{R} (X,Y,AZ,W) = & \overset{g}{R} (X,Y,AZ,W) + \frac{1}{2}   ( \overset{0}{D} (X,Y) - \overset{0}{D} (Y,X)  )g(Z,W)    -  \frac{1}{2} \overset{0}{D} (X,AZ) F(Y,W) 
				\\ & + \frac{1}{2} \overset{0}{D} (Y,AZ) F(X,W)  + \frac{1}{4} \pi(AZ) (\pi (Y) g(X,W) - \pi (X) g(Y,W)).
			\end{split}
		\end{align*}
		Adding the previous equations and using equations (\ref{eq:D0}) and (\ref{eq:relation5forRg}), we obtain
		\begin{equation*}
			\begin{split}
				\overset{0}{R} (X,Y,Z,AW) + \overset{0}{R} (X,Y,AZ,W)  = & -\frac{1}{2} ((\overset{g}{\nabla}_{X} \pi)(Z) + \pi(X) \pi(AZ) + \frac{1}{2} \pi(AX) \pi(Z) ) g(Y,W) 
				\\ & + \frac{1}{2} ((\overset{g}{\nabla}_{Y} \pi)(Z) + \pi(Y) \pi(AZ) + \frac{1}{2} \pi(AY) \pi(Z) ) g(X,W)
				\\ & - \frac{1}{2} ((\overset{g}{\nabla}_{X} \pi)(AZ) - \pi(X) \pi(Z) + \frac{1}{2} \pi(AX) \pi(AZ) ) F(Y,W) 
				\\ & + \frac{1}{2} ((\overset{g}{\nabla}_{Y} \pi)(AZ) - \pi(Y) \pi(Z) + \frac{1}{2} \pi(AY) \pi(AZ) ) F(X,W). 
			\end{split}
		\end{equation*}
		If we assume that 
		\begin{equation*}
			(\overset{g}{\nabla}_{X} \pi)(Y) + \pi(X) \pi(AY) + \frac{1}{2} \pi(AX) \pi(Y) = 0
		\end{equation*}
		then it holds
		\begin{equation*}
			\overset{0}{R} (X,Y,Z,AW) = - \overset{0}{R} (X,Y,AZ,W).
		\end{equation*}
	\end{proof}
	
	\section{Some identities obtained from $\overset{\theta}{H}$ tensors}
	
	Based on the results above, we can see that only tensor $\overset{4}{H}$ is equivalent to the well-known Weyl projective curvature tensor. By combining the remaining tensors $\overset{\theta}{H}$, $\theta=0,1,2,3,5$, we will obtain some identities for the Weyl projective curvature tensor and the holomorphically projective curvature tensor. First, we will present Weyl projective curvature tensor as a linear combination of tensors $\overset{\theta}{H}$.
	
	\begin{theorem}\label{thm:identitiesforWgKM}
		Let $(\mathcal{M}, g, A)$ be a K\"ahler manifold with a quarter-symmetric metric $A$-connection (\ref{eq:Q-S-Gmetric}). The following relations hold
		\begin{align*}
			4\overset{0}{H} (X,Y)Z - 2 \overset{1}{H} (X,Y)Z - \overset{2}{H} (X,Y)Z & = \overset{g}{W} (X,Y)Z, \\
			2\overset{5}{H} (X,Y)Z - \overset{1}{H} (X,Y)Z + \overset{1}{H} (Y,Z)X & = \overset{g}{W} (X,Z)Y,
		\end{align*}
		where $\overset{0}{H}$, $\overset{1}{H}$, $\overset{2}{H}$, $\overset{5}{H}$ are given by (\ref{eq:K0XYZgmAHM}), (\ref{eq:K1XYZgmAHM}), (\ref{eq:K2XYZgmAHM}), (\ref{eq:K5XYZgmAHM}), respectively.
	\end{theorem}
	\begin{proof}
		With help of equations (\ref{eq:R1XYZgmAHM2}) and (\ref{eq:R5XYZgmAHM2}), we have
		\begin{equation*}
			2\overset{5}{H} (X,Y)Z - \overset{1}{H} (X,Y)Z + \overset{1}{H} (Y,Z)X  = \overset{g}{R} (X,Y)Z + \overset{g}{R} (Y,Z)X + \frac{1}{n-1} ( \overset{g}{R}ic (X,Y)Z - \overset{g}{R}ic (Y,Z)X ). 
		\end{equation*}
		By using the first Bianchi identity and skew-symmetric property of Riemannian curvature tensor $\overset{g}{R}$, we get
		\begin{equation*}
			\begin{split}
				2\overset{5}{H} (X,Y)Z - \overset{1}{H} (X,Y)Z + \overset{1}{H} (Y,Z)X & = \overset{g}{R} (X,Z)Y + \frac{1}{n-1} ( \overset{g}{R}ic (X,Y)Z - \overset{g}{R}ic (Z,Y)X ) \\
				& = \overset{g}{W} (X,Z)Y.
			\end{split}
		\end{equation*}
	\end{proof}
	
	If we use equation (\ref{eq:Riccihybrid}), then from (\ref{eq:H0Rg}), we get 
	\begin{equation}\label{eq:K0PgWgKM}
		\overset{0}{H} (X,Y)Z = \frac{n+2}{4}  \overset{g}{P}(X,Y)Z - \frac{n-2}{4} \overset{g}{W}(X,Y)Z,
	\end{equation}
	where $\overset{g}{W}$ is the Weyl projective curvature tensor (\ref{eq:Weylptensor}) and $\overset{g}{P}$ is the holomorphically projective curvature tensor given by equation (\ref{eq:holomorphicallytensorKahler}). From the previous equation, we can conclude the following.
	\begin{theorem}
		Let $(\mathcal{M}, g, A)$ be a K\"ahler manifold with a quarter-symmetric metric $A$-connection (\ref{eq:Q-S-Gmetric}). If tensor $\overset{0}{H}$, given by (\ref{eq:K0XYZgmAHM}), vanishes, then it holds that
		\begin{equation*}
			\overset{g}{P}(X,Y)Z = \frac{n-2}{n+2} \overset{g}{W}(X,Y)Z.
		\end{equation*}
	\end{theorem}
	
	From equations (\ref{eq:R1XYZgmAHM2}) and (\ref{eq:R2XYZgmAHM3}), we obtain identity
	\begin{equation*}
		2\overset{1}{H} (X,Y)Z + \overset{2}{H} (X,Y)Z = (n+2)\overset{g}{P} (X,Y)Z - (n-1) \overset{g}{W} (X,Y)Z
	\end{equation*}
	from which we conclude that the following statement holds.
	\begin{theorem}
		Let $(\mathcal{M}, g, A)$ be a K\"ahler manifold with a quarter-symmetric metric $A$-connection (\ref{eq:Q-S-Gmetric}). If tensors $\overset{1}{H}$ and $\overset{2}{H}$, given by (\ref{eq:K1XYZgmAHM}) and (\ref{eq:K2XYZgmAHM}), respectively, vanish, then it holds that
		\begin{equation*}
			\overset{g}{P} (X,Y)Z = \frac{n-1}{n+2} \overset{g}{W} (X,Y)Z.
		\end{equation*}
	\end{theorem}
	
	Using theorems \ref{thm:H4KM} and \ref{thm:identitiesforWgKM}, based on equation (\ref{eq:K0PgWgKM}), we can represent the  holomorphically projective curvature tensor as a linear combination of tensors $\overset{\theta}{H}$, $\theta=0,1,\dots,5$.
	\begin{corollary}
		Let $(\mathcal{M}, g, A)$ be a K\"ahler manifold with a quarter-symmetric metric $A$-connection (\ref{eq:Q-S-Gmetric}). The following relations hold
		\begin{align*}
			\overset{g}{P} (X,Y)Z = & \frac{4}{n+2}\overset{0}{H} (X,Y)Z +\frac{n-2}{n+2} \overset{4}{H} (X,Y)Z, \\
			\overset{g}{P} (X,Y)Z = & \frac{4(n-1)}{n+2}\overset{0}{H} (X,Y)Z - \frac{2(n-2)}{n+2} \overset{1}{H} (X,Y)Z - \frac{n-2}{n+2} \overset{2}{H} (X,Y)Z, \\
			\overset{g}{P} (X,Y)Z = & \frac{4}{n+2}\overset{0}{H} (X,Y)Z +\frac{n-2}{n+2} (	2\overset{5}{H} (X,Z)Y - \overset{1}{H} (X,Z)Y + \overset{1}{H} (Z,Y)X ).
		\end{align*}
		
	\end{corollary}

	\section{Conclusion and further work}
Observing a K\"ahler manifold with a quarter-symmetric metric $A$-connection, we determined a tensor that are independent of generator $\pi$. By using newly obtained tensors $\overset{\theta}{H}$, $\theta=0,1,\dots,5$, we established some relationships between the Weyl projective curvature tensor and the holomorphically projective curvature tensor. Also, we presented them as a linear combination of tensors $\overset{\theta}{H}$. Analogously, the identities for the second holomorphically projective curvature tensor obtained by M. Prvanovi\'c in \cite{prvanovic2005} can be determined.
	
On the other hand, we observed the case when $\overset{g}{\nabla} \pi$ and $\pi\otimes\pi$ are hybrid tensors and we determined which properties are satisfied by all linearly independent curvature tensors.
	
In future work, we will try to find some more properties of the tensors $\overset{\theta}{H}$, as well as their application. This research on the quarter-symmetric connection will be continued on an almost para-Hermitian and on a para-K\"ahler manifold.

	\section{Acknowledgement}
	The financial support of this research by the projects of the Ministry of Education, Science and Technological Development of the Republic of Serbia (project no. 451-03-9/2021-14/200124 for Milan Lj. Zlatanovi\'c and project no. 451-03-9/2021-14/200123 for Miroslav D. Maksimovi\'c) and by project of Faculty of Sciences and Mathematics, University of Pri\v stina in Kosovska Mitrovica (internal-junior project IJ-0203).

\end{document}